\documentclass[12pt]{amsart}

\usepackage{tikz}
\usepackage{wrapfig}
\usetikzlibrary{positioning}
\usepackage{tikz-cd}
\tikzset{main node/.style={circle,fill=black!3,draw,minimum size=0.6cm,inner sep=0pt}}
\usetikzlibrary{matrix,arrows,decorations.pathmorphing}
\usepackage{centernot}
\usepackage{indentfirst}
\usepackage{stackengine}
\usepackage[T1]{fontenc}

\font\zzzsy =cmsy10 at 14pt
\font\zzzsyb =cmsy10 at 20pt
\def\bigtimes{\mathop{\mathchoice{\hbox{\zzzsyb\char2}}{\hbox{\zzzsy\char2}}{\scriptstyle\times}{\scriptscriptstyle\times}}}

\makeatother

\usepackage{enumerate}
\usepackage[all]{xy}  
\usepackage{graphicx}
\usepackage[headings]{fullpage}
\usepackage{amsmath,amsthm,amssymb,amsfonts,latexsym,mathrsfs}
\usepackage{color}
\usepackage{ifthen}
\usepackage[colorlinks,pagebackref=true,pdftex]{hyperref}
\usepackage[normalem]{ulem}
\usepackage{comment}

\makeatletter

\@addtoreset{equation}{section}
\def\theequation{\thesection.\@arabic \c@equation}
\def\@citecolor{blue}
\def\@urlcolor{blue}
\def\@linkcolor{blue}
\def\theenumi{\@roman\c@enumi}

\makeatother

\theoremstyle{definition}
\newtheorem{theorem}{Theorem}[section]
\newtheorem{lemma}[theorem]{Lemma}
\newtheorem{corollary}[theorem]{Corollary}
\newtheorem{proposition}[theorem]{Proposition}
\newtheorem*{claim*}{Claim}

\theoremstyle{definition}
\newtheorem{remark}[theorem]{Remark}

\newtheorem{remarks}[theorem]{Remarks}

\newtheorem{example}[theorem]{Example}

\newtheorem{definition}[theorem]{Definition}

\def\NZQ{\mathbb}               

\def\ZZ{{\NZQ Z}}
\def\QQ{{\NZQ Q}}

\def\CC{{\NZQ C}}
\def\opn#1#2{\def#1{\operatorname{#2}}} 
\opn\chara{char}
\opn\length{\ell}
\opn\projdim{proj\,dim}
\opn\depth{depth}
\opn\reg{reg}
\opn\lreg{lreg}
\opn\sat{^{sat}}
\opn\lex{^{lex}}
\opn\Ker{Ker}
\opn\Coker{Coker}
\opn\Im{Im}
\opn\Hom{Hom}
\opn\Tor{Tor}
\opn\Ext{Ext}
\opn\End{End}
\opn\Aut{Aut}
\opn\id{id}
\opn\GL{GL}

\renewcommand{\leq}{\leqslant}
\renewcommand{\geq}{\geqslant}

\renewcommand{\zeta}{\texttt{\large{Z}}}
\let\:=\colon

\let\lra\longrightarrow

\newcommand{\mb}[1]{\mathbb{#1}}
\newcommand{\mc}[1]{\mathcal{#1}}
\newcommand{\mf}[1]{\mathfrak{#1}}
\newcommand{\tf}[1]{\textbf{#1}}

\newcommand{\dbrightarrow}[4]{\xymatrix{#1\ar@<2pt>[r]^{#3} \ar@<-2pt>[r]_{#4} & #2}}
\newcommand{\dbleftarrow}[4]{\xymatrix{#1 & #2\ar@<2pt>[l]^{#3} \ar@<-2pt>[l]_{#4}}}
\newcommand{\looparrow}[2]{\xymatrix{#1\ar@(ul,ur)^{#2}}}
\newcommand{\eloop}[1]{\xymatrix@1{#1\ar @{-} @(ul,ur) }}

\DeclareMathOperator{\Eq}{Eq}

\DeclareMathOperator{\Ann}{Ann}

\title{Zero-divisor graphs and zero-divisor functors}
\author{Enrico Sbarra}
\address{Enrico Sbarra - Dipartimento di Matematica - Universit\`a degli Studi di Pisa - Largo Bruno Pontecorvo 5 - 56127 Pisa - Italy}
\email{enrico.sbarra@unipi.it}
\author{Maurizio Zanardo}
\address{Maurizio Zanardo - Dipartimento di Matematica - Universit\`a degli Studi di Pisa - Largo Bruno Pontecorvo 5 - 56127 Pisa - Italy} 
\email{m.zanardo@studenti.unipi.it}

\subjclass[2010]{Primary 05C25; Secondary 18A30, 13P25}

\keywords{Zero-divisor graph, zero-divisor functor, staircase graph, artinian rings}
\begin{document}
\begin{abstract}
Inspired by a very recent work of A. \DJ uri\'c, S. Jev\dj eni\'c and N. Stopar, we introduce a new definition of zero-divisor graphs attached to rings, that includes all of the classical definitions already known in the literature. We provide an interpretation of such graphs as images of a functor, that we call zero-divisor functor and which is associated with a family of special equivalence relations fixed beforehand. We thus recover and generalize many known results for zero-divisor graphs and provide a framework which might be useful for further investigations on this topic. 
\end{abstract}

\maketitle
\smallskip

\section*{Introduction}
Throughout this paper a ring will be a unitary commutative ring. 
The notion of zero-divisor graph was introduced for the first time by Beck in 1988, see  \cite{B88}, with the purpose of investigating the structure of the underlying ring. Given a ring $A$, Beck's zero-divisor graph  $G(A)$ is a graph whose vertex set $V(A)$ is the set of all elements of  $A$ and the set of edges $E(A)$ is determined by the following zero-divisor relation on the elements of $A$: two vertices $a,b\in V(A)=A$ are adjacent if and only if $a$ and $b$ satisfy the relation $ab=0$ in $A$. The main concern in \cite{B88} was to study the chromatic and clique  numbers of $G(A)$. Shortly after, in \cite{AL99},  Anderson and Livingston introduced a slight modification of the definition of zero-divisor graph, called $\Gamma(A)$, with the purpose of simplifying $G(A)$ and its visualization without losing the relevant pieces of information therein contained; their definition of the graph is considered as classical in the literature.  Many properties of  $\Gamma(A)$, such as connectedness, see \cite{CSSS12}, and its invariants, as for instance diameter, girth, and chromatic number,  have been later studied in \cite{AM04} and \cite{AL99}. These papers also address the problem of finding rings $A$ whose zero-divisor graph  belongs to special classes of graphs, such as star graphs, complete graphs, $r$-partite complete graphs, planar graphs etc., see also  \cite{AMY03}.

Since $\Gamma(A)$ can have very many edges even when  $A$ is finite of small cardinality, in \cite{M02}  Mulay  introduced a new relation, which we denote $\asymp_A$, on the elements of $A$, defined as follows:  given $a, b\in A$, one lets $a\asymp_A b$ and, therefore, $a, b$ are adjacent as vertices of $\Gamma(A)$, if and only if  $\Ann(a)=\Ann(b)$.  The resulting graph, denoted by  $\Gamma_E(A)$, has been later referred to as {\it compressed zero-divisor graph}, see  \cite{AL12}.
Thus, the purpose of $\Gamma_E(A)$ is to simplify further the original graph by ``compressing'' it or, in other words, by diminishing the ``noise'', i.e. the redundancy, given by those elements of $A$ which essentially behave in the same fashion. Compressed zero-divisor graphs have been subsequently studied by several authors, cf. for instance \cite{SW11}, \cite{CSSS12} and \cite{AL12}.

In a very recent paper \cite{DJS21}, \DJ uri\'c, Jev\dj eni\'c and  Stopar propose  a new approach to the study  of zero-divisor graphs, by looking at functorial properties they might be endowed with. They define a functor, called $\Theta$, from the category of finite rings to that of graphs with loops; the image of a ring $A$ is a compressed zero-divisor graph $\Theta(A)$, whose set of vertices is given by  $A/\hspace{-.2cm}\sim_A$, where $a\sim_A b$ if and only if  $(a)=(b)$ is the same ideal of $A$, and where two equivalence classes $[a],[b]\in A/\hspace{-.2cm}\sim_A$, i.e. vertices of $\Theta(A)$,  are adjacent if and only if $ab=0$ in $A$.

\smallskip
The starting point of our work is to develop one of the main ideas in \cite{DJS21} and the leitmotif of the first part of this paper is to search for equivalence relations that are a good compromise between having  functorial properties and producing as little noise as possible on zero-divisor graphs, for instance preserving their compression.  From this perspective, one may wonder which are the appropriate coarsest, i.e .minimal, equivalence relations such that the associated zero-divisor functor satisfies a certain property. Another interesting question is to study what happens if we introduce a degree of freedom into the problem by letting each ring to be endowed with its own equivalence relation. 
  
We introduce some new definitions which provide the right framework to analyze and, in some cases, to provide answers to such questions; we do so in the central part of the paper which contains our main results. We have illustrated an application of our approach in the last part of the paper, where we study which rings $(A,\mc{R}_A)$ have a finite associated zero-divisor graph $\zeta(A,\mc{R}_A)$. Since every artinian ring can be written as finite product of local artinian rings, we choose as a good family of equivalence relations $\{\sim_A \}$; in this  way the associated functor preserves finite products in both directions and this allows us to provide a partial structure theorem for rings $A$ such that
$\zeta(A,\sim_A)$ is finite.

  
\smallskip
After reviewing a few preliminaries on graphs in Section \ref{PRE} and the classical definitions of zero-divisor graphs in the literature in Section \ref{ZER}.1, we look at which are the significant properties that an adjacency relation must have to define a zero-divisor graph. We call such relations zero-divisor relations, see Definition \ref{ZDR} and we observe in Proposition \ref{asymp} that a zero-divisor relation $\mc{R}_A$ is just an equivalence relation which is finer than $\asymp_A$. In this way, given a ring $A$ and a zero-divisor relation $\mc{R}_A$, we define a zero-divisor graph $\zeta(A,\mc{R}_A)$. Choosing the right $\mc{R}_A$, by means of $\zeta(A,\mc{R}_A)$ and of its strong quotients we can recover the definitions of the zero-divisor graphs studied so far. We  prove in Theorem \ref{conncompresso} that the compressed zero-divisor graphs associated with $\mc{R}_A$, cf. Definition \ref{ZDG}, are connected independently of the chosen relation, generalizing \cite[Theorem 2.3]{AL99} and \cite[Proposition 1.4]{SW11}, see also \cite{M02}.

Next, we proceed to  make $\zeta$ into a functor from the category of commutative rings $\tf{CRing}$ to that of graphs with loops $\tf{Graph}$; we do this in Section \ref{FUN}.1. We consider each object $A\in\tf{CRing}$ associated with a fixed relation $\mc{R}_A$, i.e. we fix a family $\{\mc{R}_A\}$ of equivalence relations. We define first the key property we need on a family $\{\mc{R}_A\}$ of zero-divisor relations, see Definition \ref{ffr}, and provide some examples. We then define its associated zero-divisor functor $\zeta \,:\, \tf{CRing} \lra \tf{Graph}$,\, $(A,\mc{R}_A) \mapsto \zeta(A,\mc{R}_A)$, see Definition \ref{fzd}.

\noindent We spend the rest of Section \ref{FUN} in investigating the properties of the newly defined functor $\zeta$. We first show under which conditions on $\{\mc{R}_A\}$ the associated functor $\zeta$ preserves binary products and equalizers, see Propositions \ref{cpf} and \ref{ceq}, and conclude that the families $\{=_A\}$, $\{\approx_A\}$ and $\{\sim_A\}$ preserve finite products, cf. Corollary \ref{cpfsim}, whereas $\{=_A\}$ is the only one which preserves equalizers, and thus all finite limits. Corollary \ref{cpfsim} also extends to the case of infinite rings \cite[Proposition 3.4]{DJS21}, see Corollary \ref{djsdjs}.  Section \ref{FUN}.4 contains one of the main results of this paper, Theorem \ref{ThetaA} and its consequence Corollary \ref{totale}, which we call Inversion of Product. There we show under which conditions on $\{\mc{R}_A\}$ the decomposition of $\zeta(A,\mc{R}_A)$ as a finite direct product of graphs leads to a decomposition of $A$ as a finite direct product of rings. This corrects and extends to the case of infinite rings the main result \cite[Theorem 3.6]{DJS21}.

Finally, in the last section, we apply the main results of Section \ref{FUN} to  study the case when $\zeta(A)=\zeta(A,\mc{R}_A)$  is finite, focusing on the case $\mc{R}_A=\sim_A$, which turns out to be of main interest,
as we observe at the beginning of the section. We prove in Theorem \ref{finitetheta} that $\zeta(A)$ is finite exaclty when $A$ is the direct product of an artinian PIR with a finite ring. Moreover, we prove that local artinian PIRs  are characterized as those rings $A$ for which $\zeta(A)$ is a staircase graph. Finally, thanks to the Inversion of Product Theorem, we characterize artinian PIRs as those rings whose $\zeta(A)$ is a finite product of staircase graphs, see Proposition \ref{PIRloc} and Theorem \ref{thm: caratterizzazione PIR}, which extend to the case of infinite rings the main results of \cite{DJS21}, Section 4 and 5.

\section{Preliminaries}\label{PRE}
All the rings we consider are unitary and commutative, and all rings homomorphisms  $\varphi \,:\, A \lra B$ map $1_A$ to $1_B$.  We denote by $\mc{I}_A$ the family of ideals of  $A$ and by  $\mc{PI}_A\subseteq \mc{I}_A$ the subfamily of principal ideals.
By $\mc{D}(A)$ and  $\mc{N}(A)$ we denote the set of zero-divisors and  the set of nilpotents of $A$, respectively. We also let  $\mc{D}^*(A):=\mc{D}(A)\smallsetminus\{0\}$. By $A^\times$ we denote the set of invertible elements of $A$ and  by $\Ann(a)$ the ideal $0:a$, with $a\in A$.

We recall next some definition and notation from Graph Theory, and we refer the reader to \cite{K11} for further details. A {\it graph} $G$ is an ordered couple $(V(G),E(G))$, where $V=V(G)$ is a set,  {\it the vertex set of $G$}  and $E=E(G)$ is a binary symmetric relation on  $V^2$, called {\it the edge set} of $G$.
Given a vertex  $v\in V$, we call $N(v)=\{ w\in V \,:\, (v,w)\in E\}$ {\it the set of adjacents of $v$}; we call {\it degree} of $v$, and denote it by $\deg(v)$, the cardinality of $N(v)$. We say that $v\in V$ {\it has a loop} if $v\in N(v)$ and, when this is the case we call $(v,v)\in E$ a {\it loop}.

Let  $G_1$ and $G_2$ be graphs. A {\it graph (homo)morphism} $\varphi \,:\, G_1\rightarrow G_2$ is a map $V(G_1)\to V(G_2)$ such that $\big((x,y)\in E(G_1)\Rightarrow (\varphi(x),\varphi(y))\in E(G_2)$, for all $x,y\in V(G_1)\big)$. Similarly, a {\it graph comorphism} $\varphi:G_1\rightarrow G_2$ is a map $V(G_1)\to V(G_2)$ such that

$\big((\varphi(x),\varphi(y))\in E(G_2)\Rightarrow (x,y)\in E(G_1)$, for all $x,y\in V(G_1)\big)$.  When $\varphi$ is both a graph morphism and comorphism  $\varphi$ is called {\it strong} or {\it full morphism}. We call a graph morphism $\varphi$  {\it (graph) isomorphism} if it bijective and its inverse is also a graph morphism. Observe that $\varphi$ is an isomorphism if and only if it is a bijective strong  morphism.

Given two graphs $H, G$ with  $V(H)\subseteq V(G)$, we say that $H$ is a  {\it subgraph of $G$} if the inclusion map is a graph morphism, and a subgraph  $H$ of $G$ is called {\it induced} if the inclusion map is a strong morphism. Let now  $\mc{R}$ be an equivalence relation on $V=V(G)$; we call {\it the quotient graph of $G$ by $\mc{R}$}, and we denote it by $G/\mc{R}$, the graph whose vertex set is $V/\mc{R}$ and, given $[x], [y]\in V(G/\mc{R})$, one has $([x], [y])\in E(G/\mc{R})$ if and only  there exist $x'\in [x]$ and $y'\in [y]$ with $(x',y')\in E(G)$, i.e.  $\pi_\mc{R}\,:\, G\longrightarrow G/\mc{R}$ defined by $v\mapsto [v]$ is a graph morphism.  When $\pi_\mc{R}$ is a strong morphism, we say that   $G/\mc{R}$ is a  \emph{strong quotient} of $G$.

We conclude this section by recalling some common operations defined on graphs. 
The \emph{Kronecker product}, or simply \emph{product}, of $G_1$ and $G_2$, denoted by  $G_1\times G_2,$ is  the graph whose vertex set is $V(G_1)\times V(G_2)$ and two couple of vertices  are adjacent if and only if both the first and second components of the couples are adjacent. The  \emph{coproduct} $G_1\sqcup G_2$ of two graphs $G_1$ and $G_2$ is the graph whose vertex set is the disjoint union of the vertices of $G_1$ and $G_2$ and two vertices are adjacent if and only if they are adjacent  in $G_1$ or  in $G_2$, i.e. its edge set is the disjoint union of the edges of both graphs. Finally, given two graph morphisms $\phi,\psi: G_1\to G_2$, the \emph{equalizer} $\Eq(\varphi,\psi)$ is the (strong) subgraph induced by the subset of vertex where $\varphi$ and $\psi$ coincide.

The pair of data given by the class of all graphs and graph morphisms, with composition and identities defined in the natural way, is a category, that will be  denoted by $\tf{Graph}$, with product, coproduct and equalizer defined as above,   cf. \cite[Sections 4.1 and 4.2]{K11}. 


\section{Zero-divisor graphs}\label{ZER}
There are various definitions of zero-divisor graphs attached to commutative rings which appear in the literature and we are going to review next some of the most significative. As we did above for graphs, given a ring $A$ and an equivalence relation  $\mc{R}$ on $A$, we shall denote by  $[a]$ the equivalence class of $a$ with respect to  $\mc{R}$. Two elements  $a,b\in A$ are called  \emph{associated}, and we write  $a\sim b$, if $(a)=(b)$. We say that $a,b\in A$ are \emph{strongly associated}, and we write $a\approx b$, if there exists $u\in A^\times$ such that $a=ub$. Finally, we call $a,b\in A$ \emph{equiannihilated}, and we write  $a\asymp b$, if $\Ann(a)=\Ann(b)$. It is immediate to see that  $\sim, \approx$ and $\asymp$ are indeed equivalence relations.

\subsection{Definitions in the literature}\label{DD}
Essentially, a zero-divisor graph attached to a ring $A$ is a graph whose vertex set is  $A/\mc{R}_A$, where $\mc{R}_A$ is an equivalence relation on $A$, together with an adjacency relation.

In \cite{B88},  $\mc{R}_A$ is the trivial relation, i.e. equality;  therefore, Beck's zero-divisor graph $G(A)$ has vertex set $V(G(A))=A$.
Loops are not allowed in $G(A)$; in fact, the adjacency relation  $E=E(G(A))$ is given by
$$(a,b)\in E \text{\,\,if and only if\,\,} a\neq b\,\,\text{and}\,\, ab=0.$$
In \cite{AL99}, the vertex set is just  $\mc{D}^*(A)$ and the resulting graph $\Gamma(A)$ is the induced subgraph of $G(A)$ by $\mc{D}^*(A).$

\noindent
In \cite{M02} the so-called compressed zero-divisor graph $\Gamma_E(A)$ has vertex set  $\mc{D}^*(A)/\asymp$, and the adjacency relation $E$ is defined as
$$\big([a],[b]\big)\in E\,\,\text{if and only if\,\,} a\neq b\,\,\text{and}\,\, ab=0;$$
it is a graph with no loop which is a strong quotient of $\Gamma(A)$.

\begin{center}
\includegraphics[scale=0.3]{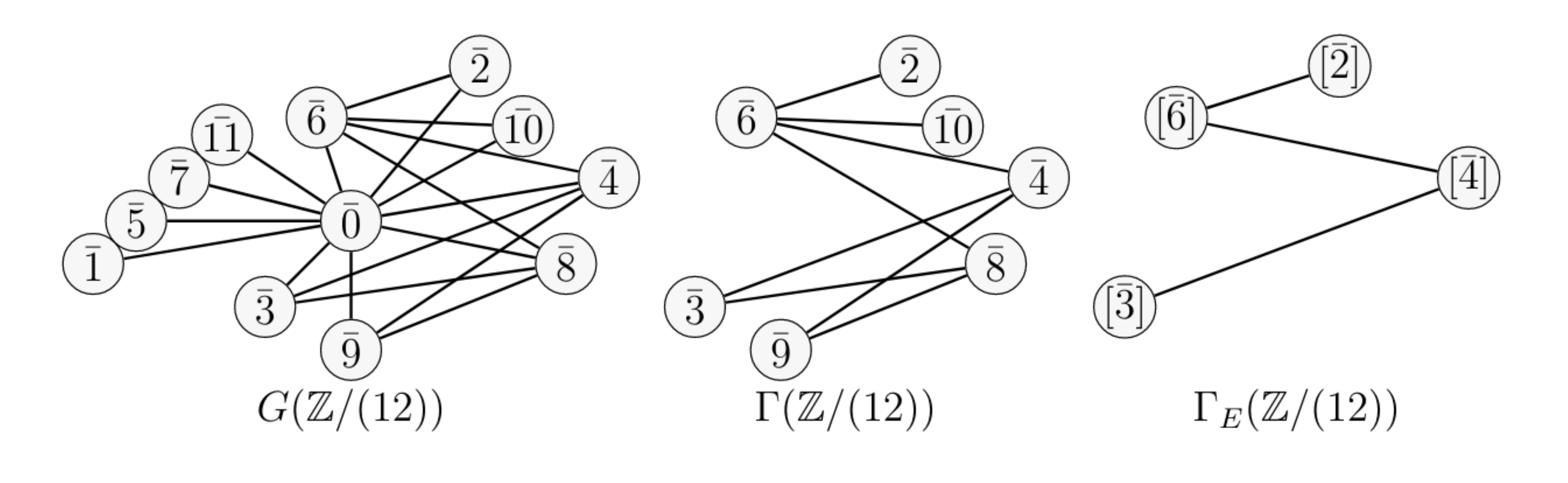}    
\end{center}

In \cite{DJS21}, functorial properties of the compressed zero-divisor graph have been studied in the case of a finite ring $A$, by means of a functor $\Theta$;  the image $\Theta(A)$ results to be a compressed zero-divisor-graph whose vertex set is  $A/\!\approx$ and the adjacency relations $E$ is defined by $$\big([a],[b]\big)\in E\,\,\text{if and only if}\,\, ab=0.$$
At the end of the same paper, the authors propose a modification, apt to treat the general case of infinite rings.

\begin{definition}[\cite{DJS21}, Definition 6.1]\label{DJStheta}
  Let  $A$ be a ring. The \emph{zero-divisor graph on principal ideals} $\Theta(A)$ has vertex set $\mc{PI}_A$ and the  adjacency relation $E$ is defined by  $$\big((a),(b)\big)\in E\,\,\text{if and only if}\,\, ab=0,$$ i.e. two principal ideals are adjacent if the product of their principal generators is zero or, equivalently if their product is the zero ideal.
\end{definition}

\begin{center}
    \includegraphics[scale=0.28]{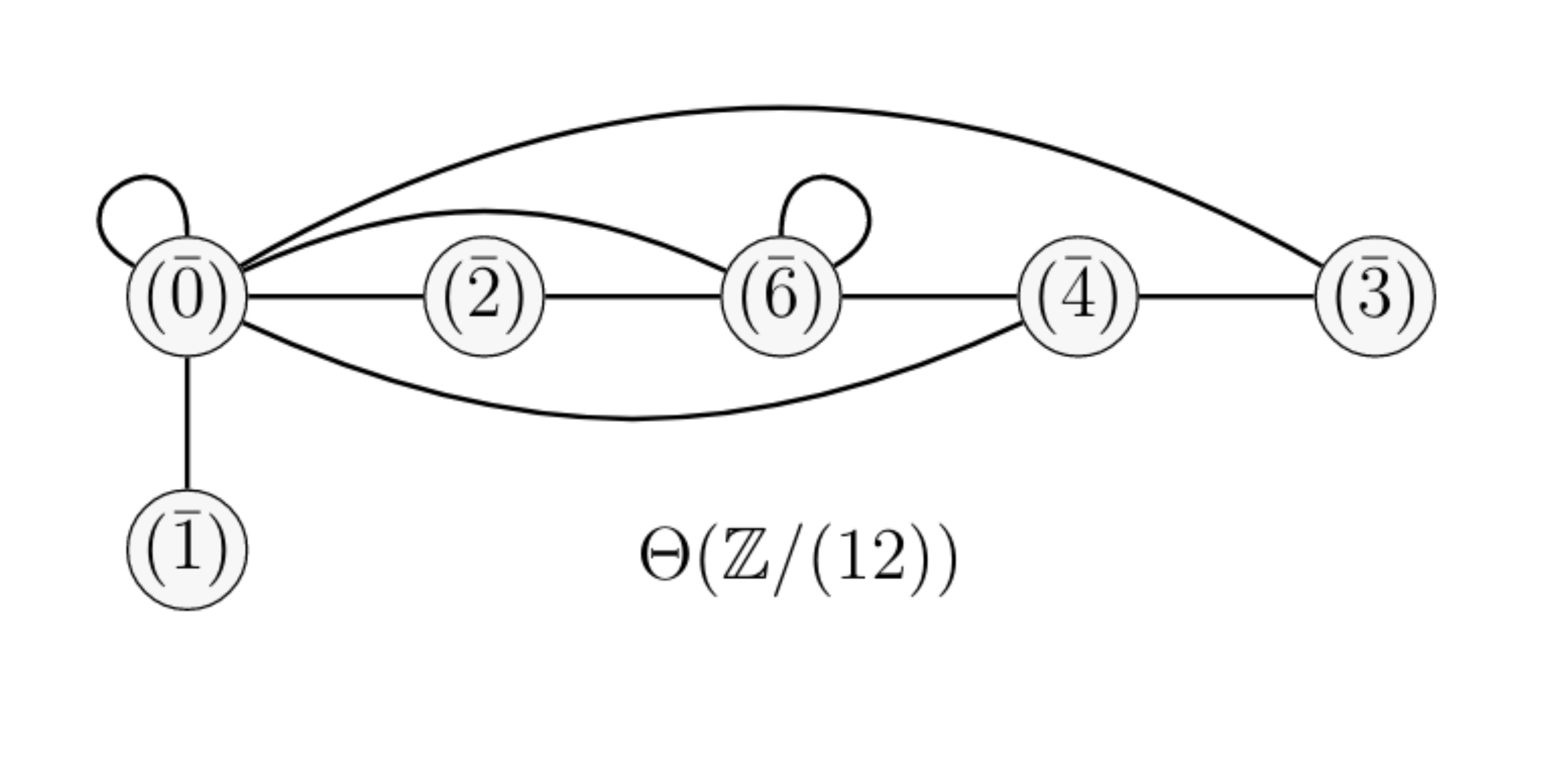}
\end{center}

\noindent With a slight abuse of notation, we shall use both $(a)$\, and \, $[a]$ to denote a vertex of $\Theta(A)$, since $(a)\mapsto [a]$ is a 1:1 correspondence between $\mc{PI}_A$\, and\, $A/\sim.$

\noindent Unlike in the other graphs introduced so far,  loops are allowed in  $\Theta(A)$, and this fact is essential in order to provide categorical properties of the functor $\Theta$ for both finite and infinite rings.

\subsection{Zero-divisor relations and a new definition}
One of the key observations in order to provide a generalization is that all of the adjacency relations which have been considered so far all share the same ``minimal''  property, which lead us to the following definition.

\begin{definition}\label{ZDR}
  An equivalence relation  $\mc{R}$ on $A$ is called  \emph{zero-divisor relation} if, taken any  $a\mc{R}a'$ e $b\mc{R}b'$, it holds that  $\left(ab=0\Rightarrow a'b'=0\right)$.
\end{definition}

Observe that an equivalence relation $\mc{R}$ is zero-divisor if and only if the adjacency relation $E$ on $A/\mc{R}$, defined by $([a],[b])\in E\!\iff\!ab=0$, is well defined. 
This fact suggests the choice of the name and puts us in a position to provide the announced definition.

\begin{definition}\label{ZDG}
 A \emph{zero-divisor graph on $A$} is a graph whose vertex set is  $A/\mc{R}$, for some zero-divisor relation $\mc{R}$ on $A$, and $[a]$ and $[b]$ are adjacent if and only if $ab=0$; we denote it by $\zeta(A,\mc{R})$.

  Given any subset $X\subset A $, we also let  $\zeta(X,\mc{R})$ be the subgraph of  $\zeta(A,\mc{R})$ induced by $\{[a]\mid a\in X\}$.

\end{definition}

\noindent
When there is no ambiguity, we shall denote $\zeta(A,\mc{R})$, resp. $\zeta(X,\mc{R})$, simply by $\zeta(A)$, resp. $\zeta(X)$. A case of special interest is when $X=\mc{D}^*(A)$,  since  the essential part of a zero-divisor graph is given by the equivalence classes of the elements of  $X$.

By the very definitions, we can capture the graphs $G(A)$,\, $\Gamma(A)$,\, $\Gamma_E(A)$ by means of $\zeta(A,=)$,\, $\zeta(D^*(A),=)$,\, $\zeta(D^*(A),\asymp)$ respectively, removing their loops. Moreover, for any finite ring $A$, we have
$\Theta(A)=\zeta(A,\approx)$

Among all zero-divisor relations on $A$, the relation $\asymp$ turns out to be the coarsest, and this somewhat suggests the appellative compressed of the graph $\Gamma_E(A)$. 
Also, maybe surprisingly, the viceversa holds, i.e. if $\mc{R}$ is finer than the relation $\asymp$, then it is a zero-divisor relation.
We record this fact in the next result. 

\begin{proposition}\label{asymp}
  Let $\mc{R}$ be a an equivalence relation on $A$; then $\mc{R}$ is zero-divisor if and only if $\mc{R}$ is finer than $\asymp$, i.e. for all $a$ and $b$, $a\mc{R}b$ implies $a\asymp b$. 
\end{proposition}
\begin{proof}
Let us take $a\mc{R}b$, for a given zero-divisor relation $\mc{R}$ on $A$ and, by symmetry, it is sufficient to  prove that $\Ann(a)\subseteq\Ann(b)$.
Given any $c\in\Ann(a)$, we have that  $ca=0$ and, since  $a\mc{R}b$ and $c\mc{R}c$, it follows that  $cb=0$.\\
Viceversa, if $a\mc{R}a'$ and $b\mc{R}b'$  for a finer relation than $\asymp$ then $\Ann(a)=\Ann(a')$ and $\Ann(b)=\Ann(b')$. So, if $ab=0$, then $b\in\Ann(a)=\Ann(a')$, i.e. $a'b=0$; therefore $a'\in\Ann(b)=\Ann(b')$, i.e. $a'b'=0$.
\end{proof}


A consequence of the previous proposition is that, for all zero-divisor relations on $A$, it holds that $[0]=\{0\}$, because $A=\Ann(0)=\Ann(a)$ implies $a=1\cdot a=0$. It is also easy to see that $=$, $\approx$ and $\sim$ are zero-divisor relations, because $a=a'$ implies $a\approx a'$, which in turn yields  $(a)=(a')$ and, thus,   $\Ann(a)=\Ann(a')$ for all $a,a'\in A.$\\
Also observe that, in general,  $\sim$ is not the same as $\asymp$; take for instance the ring $\ZZ$, where we have $2\asymp 1$ but $(2)\neq (1)=\ZZ.$

\begin{remark}\label{coarse&quot}
Observe that, given zero-divisor relations $\mc{R}, \mc{R}'$ on $A$ such that $\mc{R}$ is finer than $\mc{R}'$, then $\zeta(A,\mc{R}')$ is a strong quotient of $\zeta(A,\mc{R})$. 
\end{remark}

Many known facts about zero-divisor graphs can be generalized in terms of $\zeta$, as for instance the following, cf. \cite{AL99} Theorem 2.3, \cite{M02}, \cite{SW11} Proposition 1.4; see also \cite{SZ22} for more results of this type.

  \begin{theorem}\label{conncompresso}
Let  $\mc{R}$ be a  zero-divisor relation on $A$; then the subgraph $\zeta(\mc{D}^*(A),\mc{R})$, when not empty, is connected with diameter $\leq 3$.
\end{theorem}
\begin{proof}
  Consider the graph $G$ which has the same vertex set as  $\zeta(\mc{D}^*(A),\mc{R})$ and with edge set obtained by $E(\zeta(\mc{D}^*(A),\mc{R}))$ by removing all the loops.
  Observe that, without loss of generality we can prove the statement for $G$, since adding loops does not ruin connectedness nor affects the diameter. Since $G$ is a strong quotient of $\Gamma$ by Remark \ref{coarse&quot}, the conclusion follows from \cite{AL99}, Theorem 2.3.
\end{proof}

\section{Zero-divisor functors}\label{FUN}
We denote by \textbf{CRing} the category of commutative rings and we recall that the  categorical product in \textbf{CRing} is the direct product of rings, whereas the categorical coproduct  is given by the underlying tensor product of abelian groups, equipped with its canonically induced commutative ring structure.

In this section we generalize to infinite rings the ideas of  \cite{DJS21}, and prove some properties of the map $(A,\mc{R})\mapsto \zeta(A,\mc{R})$ that make it a functor from the category $\tf{CRing}$ to the category $\tf{Graph}$, by fixing beforehand a family of equivalence relations indexed by the class of all rings.  Some conditions on the family of relations we consider are needed, for $\zeta$ to behave well on ring homomorphisms. 

\subsection{Functorial families of equivalence relations}
Let now $\{\mc{R_A}\}$, where $A$ varies in $\tf{CRing}$, denote a collection of equivalence relations.

\begin{definition}\label{ffr} We say that $\{\mc{R_A}\}$ is a \emph{functorial family} in $\tf{CRing}$ if, for all morphisms  $f \,:\, A\rightarrow B$ in $\tf{CRing}$, it holds that  $a\mc{R}_Aa'\Rightarrow f(a)\mc{R}_Bf(a')$.
\end{definition}

A collection of equivalence relations is thus functorial in \tf{CRing}  if every ring homomorphism maps equivalence classes to equivalence classes. The appellative functorial in the previous definition is motivated by the following fact.

\begin{proposition}\label{vertexfunct}
Let $\{\mc{R}_A\}$ be a family of equivalence relations; then $\{\mc{R}_A\}$ is functorial in \tf{CRing} if and only if there exists a functor  $F:\tf{CRing}\rightarrow\tf{Set}$ such that $F(A)=A/\mc{R}_A$  and  $F(f)([a])=[f(a)]$ for all rings $A$ and homomorphisms $f \,:\, A\rightarrow B$.
\end{proposition}
\begin{proof}
The map $F(f)$ is well-defined if and only if $\{\mc{R}_A\}$ is functorial. The equalities $F(\id_A)=\id_{F(A)}$ and $F(f\circ g)=F(f)\circ F(g)$ are a direct consequence of the definition of $F(f)$.
\end{proof}

\begin{example}\label{ffexamp}
The following  are examples of functorial families of relations.

\begin{enumerate}[1.]
\item  $\{\mc{R}_A\}=\{=_A\}$, the \emph{trivial} functorial family;

\item  $\{\mc{R}_A\}=\{\approx_A\}$, since,  if $a_1=ua_2$ where $u\in A^\times$, then $f(a_1)=f(u)f(a_2)$ and $f(u)$ is a unit for every ring homomorphism $f$;

  \item  $\{\mc{R}_A\}=\{\sim_A\}$, since, if $(a_1)=(a_2)$ then $(f(a_1))=(f(a_2))$ holds, for all homomorphisms $f$.
\end{enumerate}
\end{example}

\begin{example}\label{ffexample2}
1. Let $\{\mc{R}_A\}$ be a family of equivalence relations such that $A/\mc{R}_A$ has only one equivalence class for each ring $A$; then $\{\mc{R}_A\}$ is a functorial family, which we call \emph{totally degenerate}.

\smallskip
\noindent
2. The family $\{\mc{R}_A\}=\{\asymp_A\}$ is not functorial. To this end, take  the  projection  $\pi \,:\,\ZZ\to\ZZ/(2)$, $n \mapsto \bar{n}$; we then have $\Ann(\bar{1})=\Ann(\bar{2})=0$ but $\Ann(\bar{1})=0\neq\ZZ/(2)=\Ann(\bar{2})$.
\end{example}

All the families of Example \ref{ffexamp} are in a sense tame, since they consist essentially of the same relation defined on all of the rings. In general, things can be more intricated, as in the following example.

\begin{example}\label{ffexample3}
Consider $\{\mc{R}_A\}=\{\sim_A\}$ and $\{\mc{R}'_A\}=\{\approx_A\}$, which are  functorial as we have seen in Example \ref{ffexamp} and notice that, for all $A$, we have $\mc{R}'_A$ is finer than $\mc{R}_A$. Also observe that the set of nilpotents $\mc{N}(A)$ can be partioned as $\mc{N}(A)=\bigsqcup_{a\in\mc{N}(A)}[a]_{\mc{R}_A}$.  We define now the family $\{\mc{R}''_A\}$ as follows: for all $A\in\tf{CRing}$, $a,a'\in A$ we let  $$a\mc{R}''_A a' \text{ if and only if } \big(a\mc{R}_Aa' \text{ when } a,a'\in\mc{N}(A) \text{  and } a\mc{R}'_Aa' \text{  otherwise}\big).$$
  Since $\{\mc{R}'_A\}$ is finer than $\{\mc{R}_A\}$ and all homomorphism $f:A\to B$ are such that  $f(\mc{N}(A))\subseteq\mc{N}(B)$, the family $\{\mc{R}''_A\}$ is a functorial family too. 
\end{example}

We say that a family of equivalence relations $\{\mc{R}_A\}$ is \emph{degenerate} if there exists a ring $A$ such that $[0]\neq\{0\}$ in $A/\mc{R}_A$ and {\it non degenerate} otherwise.

\begin{proposition}\label{finestblend}
Every non degenerate functorial family of equivalence relations $\{\mc{R}_A\}$ is such that  $\big(a\mc{R}_Ab\Rightarrow a\sim_A b$, for all $A\in\tf{CRing}\big)$, i.e. each $\mc{R}_A$ is finer that $\sim_A$.
\end{proposition}
\begin{proof}
For every ring $A$ and elements $a,b\in A$, consider the projection of $A$ onto $A/(a)$; since  $\{\mc{R}_A\}$ is functorial we have $\bar{a}\mc{R}_{A/(a)}\bar{b}$ and $\bar{a}=\bar{0}$; thus, from $[0]=\{0\}$ we get that $(b)\subseteq (a)$. By symmetry we also obtain $(a)\subseteq (b)$, as desired. 
\end{proof}

\begin{remark}
Unlike the case of zero-divisor relations, see  Proposition \ref{asymp}, the converse of the above proposition does not hold.  Take for istance the family $\{\mc{R}_A\}=\{=_A: A\in\tf{CRing}\smallsetminus\ZZ\}\cup\{\sim_\ZZ\}$ and the inclusion homomorphism $\ZZ\subseteq\QQ$. Since  $1\sim_\ZZ\!-\!1$ and $1\neq_\QQ\!-\!1$, we may conclude that  $\{\mc{R}_A\}$ is not functorial.
\end{remark}

\subsection{Definition of zero-divisor functors}
With the above definitions, by taking a \emph{zero-divisor functorial family} $\{\mc{R}_A\}$, i.e. a functorial family such that $\mc{R}_A$ is a zero-divisor relation for every $A\in\tf{CRing}$, we are now in a position to define a functor $\zeta$ that generalizes $\Theta$.  Notice that all the functorial families in the Example \ref{ffexamp} are zero-divisor.

\begin{proposition}\label{fzd}
 Let $\{\mc{R}_A\}$ be a zero-divisor functorial family in $\tf{CRing}$; then, there exists a functor $\zeta \,:\,\tf{CRing}\rightarrow\tf{Graph}$ such that $\zeta(A)$ is the zero-divisor graph $\zeta(A,\mc{R}_A)$ and $\zeta(f)([a])=[f(a)]$ for all $f \,:\,A\rightarrow B$.\\
We call $\zeta$  the \emph{zero-divisor functor associated to $\{\mc{R}_A\}$}. 
\end{proposition}
\begin{proof}
  Since  $\mc{R}_A$ is a zero-divisor relation, $\zeta(A)$ is a zero-divisor graph and, since $\{\mc{R}_A\}$ is functorial, by Proposition \ref{vertexfunct} the map $\zeta(f): V(\zeta(A))\rightarrow V(\zeta(B))$ is well-defined. Finally, it is easy to verify that $\zeta(f)$ is a graph morphism, because $ab=0$ implies $0=f(ab)=f(a)f(b)$.
\end{proof}

\noindent The functor $\zeta$ we just defined by means of Proposition \ref{fzd} generalizes the functor $\Theta$ defined in \cite[Definition 3.2]{DJS21}, over the category of finite rings, and \cite[Definition 6.1]{DJS21}, see also Definition \ref{DJStheta}. In fact, if we choose $\{\mc{R}_A\}=\{\sim_A\}$, we will have that $\zeta=\Theta$.

\begin{remark}\label{osc zdff}
By Proposition \ref{asymp}, a zero-divisor family $\{\mc{R}_A\}$ is non degenerate, since $[0]=\{0\}$ when $\mc{R}_A=\asymp_A$. Thus, by Proposition \ref{finestblend}, all  zero-divisor functorial families range  between $\{=_A\}$, which is the finest, and $\{\sim_A\}$, which is the coarsest. In particular,  the family  $\{\sim_A\}$ is the one that compresses the most each of the zero-divisor graphs $\zeta(A)$.  
\end{remark}

\begin{example}\label{usefulconsequence}
Clearly, if $f$ is an isomorphism in $\tf{CRing}$, then $\zeta(f)$ is an isomorphism in $\tf{Graph}$; therefore the rings $A=\ZZ/(4)[x]/(2x,x^2)$ and $B=\ZZ/(4)[x]/(2x,x^2-2)$ are not isomorphic, since  $\zeta(A,\sim_A)$ has five elements and $\zeta(B,\sim_B)$ has four.
\end{example}

Recall that $f \,:\, X\to Y$ is a {\it monomorphism} in a category \textbf{C} when  for every object $Z$ and every pair of morphisms $g_1, g_2 \,:\, A\to X$ it holds that $\big((fg_1=fg_2)\Rightarrow (g_1=g_2)\big)$;\, similarly, $f$ is an epimorphism in  \textbf{C} when for every object $Z$ and every pair of morphisms $g_1,g_2 \,:\, Y\to Z$ then $\big((g_1f=g_2f)\Rightarrow (g_1=g_2)\big)$.
In general, if $f$ is a monomorphism, resp. an epimorphism, in $\tf{CRing}$ then,  $\zeta(f)$ is not a monomorphism, resp.  an epimorphism of graphs.
For instance, consider again as $f$ the inclusion of $\ZZ$ in $\QQ$ which is  mono and epi in \tf{CRing} and take the functorial family $\{\sim_A\}$. Then $\zeta\,:\, \zeta(\ZZ,\sim)\lra \zeta(\QQ,\sim)$ is a morphism from an infinite graph to a graph with two vertices, which is  not a monomorphism in \tf{Graph}. Instead, if we take the family $\{=_A\}$, then  the morphism $\zeta(f)\,:\, \zeta(\ZZ,=) \lra \zeta(\QQ,=)$ is not an epimorphism in  \tf{Graph}.  
Furthermore, note that in general $\zeta$ does not reflects isomorphisms in $\tf{CRing}$ and, thus, neither mono nor epimorphisms. For instance, consider the family $\{\sim_A\}$ and the inclusion $\QQ\subseteq \QQ(\sqrt{2})$. By applying $\zeta$ we obtain an isomorphism of graphs, since $\QQ$ and $\QQ(\sqrt{2})$ are both fields,  but $\QQ(\sqrt{2})\not\simeq \QQ$ in \tf{CRing}.

We conclude this part by observing that, in general, the functor $\zeta$ is neither faithful nor full. Consider the family $\{\sim_A\}$ and the identity and conjugation automorphisms of $\CC$. Then, their images by $\zeta$ are equal; also,  not all graphs morphisms $\zeta(\CC) \lra \zeta(\CC)$ are images of rings morphisms, for  there is  the trivial graph morphism $\zeta(\CC) \to \zeta(\CC)$ that maps to $\tf{0}$  but, on the other hand,  no ring homomorphism $f\,:\,\CC \lra \CC$ is such that $f(1)=0$. 

\subsection{Finite products and finite limits. Coproducts.} Let us fix a zero-divisor functorial family $\{\mc{R}_A\}$. We are going to study next under which conditions its associated zero-divisor functor $\zeta$ preserves finite limits. To this end, it is sufficient to understand when $\zeta$ preserves the terminal object, binary products and equalizers, cf. for instance \cite{B94}, Proposition 2.8.2.

Clearly, independently of the choice of the family, $\zeta(\tf{0})$ is isomorphic to the graph composed by a single vertex with a loop, which is the terminal object in $\tf{Graph}$.

Consider now $A,B\in\tf{CRing}$, the projections $A\times B\to A$,\,\, $A\times B\to B$ and  the induced graph morphisms  $\zeta(A\times B)\to \zeta(A)$, $\zeta(A\times B)\to \zeta(B)$. By the universal property of the product in \tf{Graph}, there exists an unique canonical graph morphism $\phi\,:\,\zeta(A\times B)\to \zeta(A)\times \zeta(B)$.

\begin{proposition}\label{cpf}
  Let  $A, B$ be rings; then
  \begin{enumerate}[1.]
  \item  $\phi\, :\, \zeta(A\times B)\lra \zeta(A)\times\zeta(B)$ is a strong graph epimorphism;
    \item
     if $\{\mc{R}_A\}$ also satisfies the following property
     $$a\mc{R}_Aa', b\mc{R}_Bb' \Longrightarrow (a,b)\mc{R}_{A\times B}(a',b') \text{\,  for all\,}  a, a'\in A \text{\,  and\,}  b, b'\in B,$$ for all $A, B$ in $\textbf{CRing}$,  then $\phi$ is an isomorphism.
  \end{enumerate}
\end{proposition}

\begin{proof}
  1. By the very definition, $\phi[(a,b)]=([a],[b])$. Observe that $[(a,b)]$ and $[(a',b')]$ are adjacent in $\zeta(A\times B)$ if and only if $(a,b)(a',b')=(aa',bb')=0$, if and only if  $([a],[a'])\in E(\zeta(A))$ and $([b],[b'])\in E(\zeta(B))$; thus, $\phi$ is a strong morphism of graphs, which is evidently surjective.

  \smallskip
  \noindent
  2. It is immediate.
\end{proof}

\noindent Notice that, since a composition of strong graph epimorphisms, resp. isomorphisms, is a strong graph epimorphism, resp. isomorphism,  the previous proposition extends to arbitrary finite products.

\begin{corollary}\label{cpfsim}
Let $\{\mc{R}_A\}$ be  one of the zero-divisor functorial families $\{=_A\}$, $\{\approx_A\}$, or $\{\sim_A\}$. Then, for all $A_1,\dots,A_n\in \textbf{CRing}$ we have $\zeta(A_1\times\cdots\times A_n)\simeq\zeta(A_1)\times\cdots\times \zeta(A_n)$.
\end{corollary}
\begin{proof}
  One only needs to show that $\{\approx_A\}$ and $\{\sim_A\}$ satisfy the hypothesis of Proposition \ref{cpf}.2. Let $a\sim_A a'$ and $b\sim_B b'$; then there exist $h, h'\in A$\, and\,  $k, k'\in B$ such that  $a=ha'$, $a'=h'a$, $b=kb'$, $b'=k'b$. Thus, $(a,b)=(h,k)(a',b')$ and $(a',b')=(h',k')(a,b)$, i.e. $(a,b)\sim_{A\times B}(a',b')$.

  The proof for $\{\approx_A\}$ is utterly similar.
\end{proof}

Applying the previous Corollary to the family $\{\mc{R}_A\}=\{\approx_A\}$ in the case of finite rings we recover the following result.

\begin{corollary}[\cite{DJS21}, Proposition 3.4]\label{djsdjs}
The functor $\Theta \,:\, \tf{FinCRing} \lra \tf{Graph}$ preserves finite products. 
\end{corollary}
\begin{center}
    \includegraphics[scale=0.28]{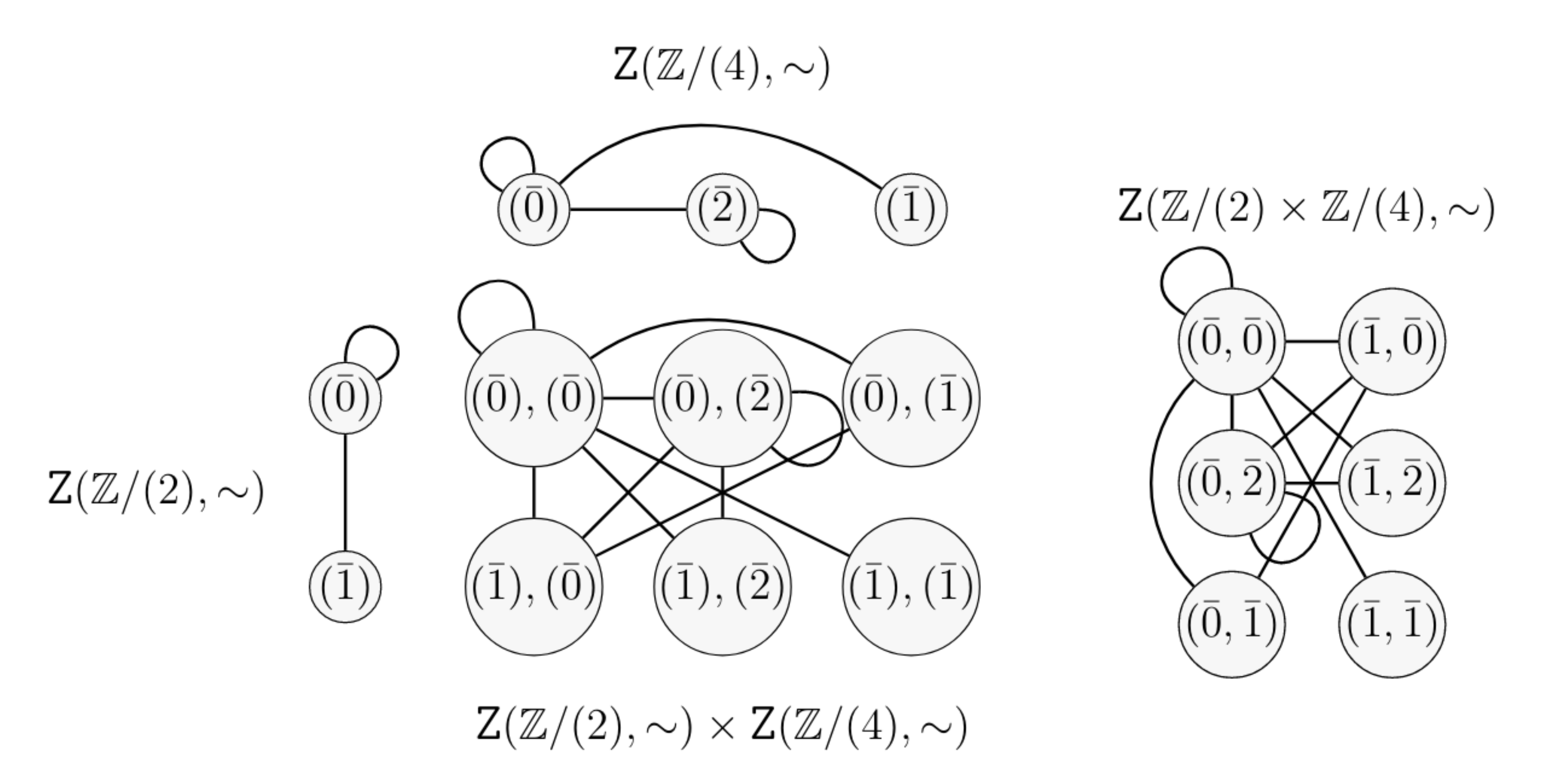}\phantom{$\zeta(\ZZ/(2),\sim)$}
\end{center}

\begin{example}\label{maunothom}
  Proceeding like in  Example \ref{ffexample3}, consider the functorial families $\{\mc{R}_A\}=\{\sim_A\}$ and $\{\mc{R}'_A\}=\{=_A\}$; notice that $A^\times=\bigsqcup_{a\in A^\times}[a]_{\mc{R}_A}$, and define $\mc{R}''_A$ by means of $\{\mc{R}_A\}$ and $\{\mc{R}'_A\}$. In the proof of Corollary \ref{cpfsim} we have seen that $\{\mc{R}_A\}$ and $\{\mc{R}'_A\}$ verify the hypothesis of Proposition \ref{cpf}.2, nonetheless $\{\mc{R}''_A\}$ does not. Take for instance $A_1=A_2=\ZZ$; then $1\mc{R}''_\ZZ\small{-}1$, $2\mc{R}''_\ZZ 2$, but it does not hold that  $(1,2)\mc{R}''_{\ZZ\times\ZZ} (\small{-}1,2)$.
  Applying the associated zero-divisor functor $\zeta$ we find out that $\zeta(\ZZ/(2)\times \ZZ/(4))$ has seven vertices, whereas $\zeta(\ZZ/(2))\times\zeta(\ZZ/(4))$ has six and, hence, they are not isomorphic. 
\end{example}

We observe that, in general, $\zeta$ does not reflects direct products. Consider for instance two morphisms $\QQ\times\ZZ/(2)\longrightarrow\QQ(\sqrt{2})$ and $\QQ\times\ZZ/(2)\longrightarrow \ZZ/(2)$; then $\zeta(\QQ(\sqrt{2})\times\ZZ/(2),\sim)\simeq\zeta(\QQ(\sqrt{2}),\sim)\times\zeta(\ZZ/(2),\sim)\simeq\zeta(\QQ,\sim)\times\zeta(\ZZ/(2),\sim)$ in \tf{Graph}, whereas $\QQ\times\ZZ/(2)\not\simeq\QQ(\sqrt{2})\times\ZZ/(2)$ in \tf{CRing}.
\smallskip

Now we look at equalizers. Consider $A, B\in\tf{CRing}$\, and\, $f, g \,:\,A\to B$ morphisms in $\tf{CRing}$. Recall that the \emph{equalizer} $\Eq(f,g)$  is the subring $\ker(f-g)\subseteq A$. Therefore, we have an inclusion map $\Eq(f,g)\to A$ that  induces a morphism $\zeta(\Eq(f,g))\to\zeta(A)$ in \tf{Graph}, which composed by  $\zeta(f)$ or by $\zeta(g)$ returns the same morphism $\zeta(\Eq(f,g))\to \zeta(B)$ and, thus, we can consider the unique canonical graph morphism $\psi:\zeta(\Eq(f,g))\to\Eq(\zeta(f),\zeta(g))$ given by the universal property of the equalizer  in \tf{Graph}.

\begin{proposition}\label{ceq}
  Let $A,B$ rings and $f,g:A\to B$ ring homomorphisms; then 
$\psi:\zeta(\Eq(f,g))\to \Eq(\zeta(f),\zeta(g))$  is a strong graph morphism.
  Moreover,
  \begin{enumerate}[1.]
\item  if $\{\mc{R}_A\}$ satisfies the following property: for all $A, B$ in $\textbf{CRing}$ and $f,g:A\to B$,
$$\big(a\mc{R}_A a' \Longrightarrow a\mc{R}_{\Eq(f,g)}a' \big)\text{\,  for all\,}  a,a'\in \Eq(f,g), $$  then $\psi$ is a graph monomorphism;

  \item if $\{\mc{R}_A\}$ also satisfies the following property: for all $A, B$ in $\textbf{CRing}$ and $f,g:A\to B$
$$\big(f(a)\mc{R}_B\;g(a) \Longrightarrow \exists\;a'\in\Eq(f,g) \text{ such that } a\mc{R}_Aa' \big)\text{\,  for all\,}  a\in A, ,$$ then $\psi$ is a graph epimorphism.
  \end{enumerate}
\end{proposition}
\begin{proof} By construction $\psi$ is necessarily defined as
$\psi([a]_{\Eq(f,g)})=[a]_{A}$, and it is immediately seen that $\psi$  is a strong graph morphism. The hypotheses in 1, resp. 2, are exactly what is needed for $\psi$ to be injective, resp. surjective.
\end{proof}

An example of a zero-divisor functorial family that satifies such properties is $\{=_A\}$, but not every zero-divisor functorial family is such; for instance,   the family $\{\sim_A\}$ does not preserves equalizers, as observed in \cite{DJS21}, Example 3.5. We see next that $\{=_A\}$ is indeed the only zero-divisor functorial family that preserves them.

\begin{proposition}\label{finestblendEq}
Let $\{\mc{R}_A\}$ be a zero-divisor functorial family; then its associated zero-divisor functor $\zeta$ preserves equalizers if and only if $\mc{R}_A$ is  $=_A$ for all $A\in \tf{CRing}$.
\end{proposition}

\begin{proof}
For every ring $A$ and elements $a,b\in A$ such that $a\mc{R}_Ab$, consider the unique ring homomorphisms $f, g\,:\,\ZZ[x]\to A$ such that $f(x)=a$,  $g(x)=b$. Since $\zeta(f)([x])=[f(x)]=[a]=[b]=[g(x)]=\zeta(g)([x])$, the equivalence class of the polynomial $x\in\ZZ[x]$ belongs to $\Eq(\zeta(f),\zeta(g))$. By Proposition \ref{ceq} there exists $p(x)\in\Eq(f,g)$ such that $p(x)\mc{R}_{\ZZ[x]}x$. Since  $\{\mc{R}_A\}$ is non degenerate functorial, from Proposition \ref{finestblend} we have $(p(x))=(x)$ in $\ZZ[x]$, i.e. $p(x)=\pm x$. Thus $\pm x\in \Eq(f,g)$; hence $a=\pm b$.

Now, consider the projections maps $\pi_1,\pi_2\,:\, A\times A\to A$; then,  $a\mc{R}_A b$ implies that $\pi_1(a,b)\mc{R}_A\pi_2(a,b)$. Since $\zeta$ preserves all equalizers, there exists $(c,c)\in\Eq(\pi_1,\pi_2)$ such that $(c,c)\mc{R}_{A\times A}(a,b)$;  hence  $(c,c)=\pm(a,b)$, and we may conclude that $a=b$ and that $\mc{R}_A$ is $=_A$, for all $A$.
\end{proof}

As a direct consequence, we have the following.
\begin{proposition}\label{flp}
The only  zero-divisor functorial family for which its associated zero-divisor functor $\zeta$ preserves all finite limits  is the trivial functorial family.
\end{proposition}

\noindent
Thus, if $\{\mc{R}_A\}$ is not the trivial functorial family, $\zeta$ does not  have a left adjoint functor.

Finally, notice that $\zeta$ never  preserves coproducts, independently of the choice of the family $\{\mc{R}_A\}$, since a  zero-divisor graphs is always connected, $\zeta(A\otimes_\ZZ B)$ cannot be written as coproduct of graphs.
Therefore $\zeta$ does not preserve all finite colimits and it does not admit a right adjoint functor.  Observe that, for every $A,B\in\tf{CRing}$ we have a canonical graph morphism $\zeta(A)\sqcup\zeta(B)\to \zeta(A\otimes_\ZZ B)$, which is not injective, because $[0]_A$ and $[0]_B$ both map to $[0]_{A\otimes B}$, and it is easy to see that in general it is neither surjective.

\subsection{An Inversion of Product Theorem}
Given a graph $G$ and $v\in V(G)$, one can consider the subgraph of $G$ induced by $N(v)$, where $N(v)\subseteq V(G)$ is the set of adjacents of $v$. We start with the following easy lemma.

\begin{lemma}\label{ssi}
Let $A_1, A_2$ be rings and  $G_1, G_2$ the subgraphs of  $\zeta(A_1)\times \zeta(A_2)$ induced by $N\big([0],[1]\big)$ and $N\big([1],[0]\big)$, respectively; then,  $G_1\simeq \zeta(A_1)$ and $G_2\simeq \zeta(A_2)$.
\end{lemma}
\begin{proof}
Let us show for instance that $G_1\simeq \zeta(A_1)$; the other case is similar.  Recall that a vertex $\big([a],[b]\big)$ belonging $\zeta(A)\times\zeta(B)$ is adjacent to  $\big([0],[1]\big)$ if and only if  $b=0$. The claim follows directly from the fact that in a product of graphs adjacency is defined component by component. 
\end{proof}

\begin{definition}\label{orto}
We call two elements $a,b\in A$  \emph{orthogonal}, and we write $a\perp b$, if  $ab=0$\, and\, $\Ann(a)\cap \Ann(b)=0$.
\end{definition}
\noindent In relation to the study of graphs, orthogonality has been introduced in \cite{ALS03}; we  provided in the above a modification of the original definition, that will allow us to work with graphs with loops as well.  
It is easy to see that two elements $a, b\in A$ are orthogonal if and only if  $ab=0$ and $a+b$ is an $A$-regular, cf. \cite{ALS03}, Lemma 3.3.

\noindent We are now ready to prove one of our main results.

\begin{theorem}\label{ThetaA}
  Let $\{\mc{R}_A\}$ be a zero-divisor functorial family. Let also $A, B_1, B_2$ be rings such that  $\zeta(A)\simeq \zeta(B_1)\times\zeta(B_2)$, with  $\zeta(B_1), \zeta(B_2)\not\simeq\zeta(\tf{0})$; then,

  \smallskip
  \noindent
1.    there exist $a_1,a_2\in A$ such that
 \begin{enumerate}[i)]
    \item   $a_1\perp a_2$ and $a_1, a_2\in\mc{D}^*(A)$;
    \item $\zeta(B_1)\simeq\zeta(\Ann(a_2))$\, and\,  $\zeta(B_2)\simeq\zeta(\Ann(a_1))$;
    \item there exist  strong morphisms $\psi_i \,:\,\zeta(B_i)\longrightarrow\zeta(A/\Ann(a_i))$, with $i=1,2$.
\end{enumerate}

 \smallskip
 \noindent
 2. Let $i, j\in\{1, 2\}$ with $i\neq j$. If $\bar{x}\mc{R}_{A/\Ann(a_i)}\bar{x}'$ implies that $x\mc{R}_Ax'$ for all $x,x'\in\Ann(a_j)$, then $\psi_i$ is  injective;

 \smallskip
 \noindent
 3. if $\Ann(a_1)+\Ann(a_2)=(1)$, then $A\simeq A/\Ann(a_1)\times A/\Ann(a_2)$ and $\psi_1, \psi_2$ are surjective.
   \end{theorem}

\begin{proof} 1. Let  $\varphi \,:\,\zeta(A)\longrightarrow\zeta(B_1)\times\zeta(B_2)$ a graph isomorphism and let us consider $a_1, a_2\in A$ such that $[a_1]=\varphi^{-1}([1],[0])$ and $[a_2]=\varphi^{-1}([0],[1])$. Since $([1],[0]),\, ([0],[1])\in V(\zeta(B_1)\times\zeta(B_2))$ are adjacent and  $\varphi$ is a strong morphism, we have that  $[a_1]=\varphi^{-1}([1],[0])$ and $[a_2]=\varphi^{-1}([0],[1])$ are adjacent as well, i.e.  $a_1a_2=0$.

  Since $[0]$ is the only vertex adjacent to all the others, one has necessarily  $\varphi([0])=([0],[0])$. Thus, if  $a_1$  were zero, we would have that  $\varphi^{-1}([1],[0])=[a_1]=\varphi^{-1}([0],[0])$ which implies $B_1\simeq\tf{0}$ and, in turn,  $\zeta(B_1)\simeq\zeta(\tf{0}),$ a contradiction. The case $a_2=0$ is analogous and  we have  proven the first part of the claim.

Now, let  $x\in \Ann(a_1)\cap \Ann(a_2)$ and $b_1\in B_1$,  $b_2\in B_2$ such that  $\varphi([x])=([b_1],[b_2])$. Since  $[x]$ is adjacent to both $[a_1]$ and $[a_2]$, we have that  $\varphi([x])$ is adjacent to both $\varphi([a_1])=([1],[0])$ and $\varphi([a_2])=([0],[1])$. Thus, from the very definition of adjacency of the product it follows that  $b_1= 0 = b_2$, which implies that  $[x]=\varphi^{-1}([0],[0])=[0]$  and $x=0$, and we proved i).

    \smallskip
    \noindent    By our choice of $a_1, a_2$ and the fact that $\varphi$ is an isomorphism, the conclusion of ii) is a straightforward application of  Lemma \ref{ssi}.

    \smallskip
    \noindent    iii) By symmetry, we may prove the assertion for $i=1$. By what was proved in ii), we only need to show that there exists a morphism $\psi_1' \,:\,\zeta(\Ann(a_2))\longrightarrow\zeta(A/\Ann(a_1))$ which is strong. Let us then define  $\psi_1'\big([x]\big)$ as $[\bar{x}]$. The map  $\psi_1'$ results to be a strong morphism, since  $\Ann(a_1)\cap\Ann(a_2)=0$: taken $x_1, x_2\in\Ann(a_2)$, one has that  $x_1x_2=0$ if and only if $x_1x_2\in\Ann(a_1)$.
    
    \smallskip
    \noindent    2. By symmetry, we may prove the assertion for $i=1$. From what was proved in 1.ii), we only need to show that $\psi_1'$ is injective, and that is straightforward. 

    \smallskip
    \noindent
    3. By the Chinese Remainder Theorem one has $A\simeq  A/(0)\simeq A/\Ann(a_1)\times A/\Ann(a_2)$. Observe that the maps  $\psi_i'$ defined in the proof of 1.iii) are also surjective. \end{proof}

\begin{remark}\label{comments}
  The above theorem has been inspired by and extends \cite{DJS21}, Theorem 3.6. Its proof contains a mistake, since $\Theta$ is applied to an ideal of a commutative ring.
  
\begin{enumerate}[1.]
\item
  Observe that, in Theorem \ref{ThetaA}.1.ii), $\Ann(a_2)$ is determined once $a_1$ is given, since if $a,b,c\in A$ are such that $a\perp b$ and $a\perp c$ then $b\asymp c$, which is easily seen. 

\item   When $\mc{R}_A=\sim_A$, then the hypothesis of Part 2. holds true. Let $x_1, x_2\in\Ann(a_2)$ be such that $\bar{x}_1\mc{R}_{A/\Ann(a_i)}\bar{x}_2$. Since  $\sim$ is the coarsest of such relations, see Remark \ref{osc zdff},  there exist  $h, k\in A$ such that  $x_1-hx_2=0=x_2-kx_1$ in $A/\Ann(a_1)$ and, therefore,  $x_1-hx_2,x_2-kx_1\in \Ann(a_1)$. Since  $x_1, x_2\in\Ann(a_2)$, one has $x_1-hx_2=0=x_2-kx_1$ in $A$, i.e. $(x_1)=(x_2)$.
\end{enumerate}
\end{remark}

Recall that, given a ring $A$, the set  $S=A\smallsetminus\mc{D}(A)$ is multiplicatively closed; $A$ is said to be {\em total} if $A$ is isomorphic to its total ring of fractions $S^{-1}A$; this is equivalent to require that every regular element of $A$ is invertible.

\begin{corollary}\label{totale}
Let $A$ be total, $\mc{R}_A=\sim_A$ and $B_1, B_2$ rings such that  $\zeta(A)\simeq\zeta(B_1)\times\zeta(B_2)$, with  $\zeta(B_1),\zeta(B_2)\not\simeq \zeta(\tf{0})$; then  $A\simeq A_1\times A_2$, where $A_1$, $A_2$ are quotients of $A$ such that  $\zeta(A_1)\simeq \zeta(B_1)$ and $\zeta(A_2)\simeq\zeta(B_2)$.
\end{corollary}
\begin{proof}
By Theorem \ref{ThetaA}.1 it follows that there exist  $a_1,a_2\in\mc{D}^*(A)$ such that $a_1+a_2$ is a regular element and, since $A$ is total, there exists $c\in A$ such that $c(a_1+a_2)=1$. Moreover,  $a_1\in \Ann(a_2)$ and $a_2\in \Ann(a_1)$, that implies that  $\Ann(a_1)+\Ann(a_2)=(1)$; the conclusion follows now by Remark \ref{comments}.2  and Theorem \ref{ThetaA}.2 and 3.
\end{proof}

\begin{example}\label{nototal} Consider $=\mb{F}_2[x,y]/(xy)$; then $S=A\smallsetminus(\bar{x},\bar{y})$ and $S^{-1}A$ is not a total ring. It holds that  $\zeta(S^{-1}A,\sim)\simeq\zeta(\mb{F}_2[x],\sim)\times\zeta(\mb{F}_2[y],\sim)\simeq\zeta(\ZZ,\sim)\times\zeta(\ZZ,\sim)$, but $S^{-1}A$ is local and, thus, does not admit a non trivial factorization as a direct product.
\end{example}

We conclude this section with a result about $\zeta$ and localization.

\begin{proposition} Let $A$ be a ring, $S\subset A$ multiplicatively closed and  $\varphi_S \,:\,A\longrightarrow S^{-1}A$ be the canonical homomorphism; then

\begin{enumerate}[1.]
\item
if  $\frac{a}{s}\mc{R}\frac{a}{1}$ for all $a\in A$, $s\in S$, then $\zeta(\varphi_S)$ is a graph epimorphism;

\item  if $S$ consists of regular elements, then  $\zeta(\varphi_S)$ is a graph comorphism.
\end{enumerate}
\end{proposition}
\begin{proof}
  1.  Since $\frac{a}{s}\mc{R}\frac{a}{1}$ for all $a\in A$ and $s\in S$, every element in $V(\zeta(S^{-1}(A))$ is the image by $\zeta(\varphi_S)$ of a vertex in $\zeta(A)$, as desired.

  \smallskip
  \noindent
2.   Let now  $S\subseteq A\smallsetminus\mc{D}(A)$,\, and\, $a, b\in A$ such that $\frac{a}{1}\frac{b}{1}=0$; then, there exists a regular element $c$ such that $cab=0$. Therefore $ab=0$ and $\zeta(\varphi_S)$ is a comorphism.
\end{proof}

\noindent Note that the trivial family of relations does not satisfy $1.$ of the above proposition.

\begin{corollary}
Let $\{\mc{R}_A\}=\{\sim_A\}$; then, for all total rings $A\in\textbf{CRing}$ and multiplicatively closed sets $S\subseteq A$,  we have that $\zeta(\varphi_S)$ is a strong graph epimorphism. In particular,  $\zeta(S^{-1}A)$ as a strong quotient of $\zeta(A)$.
\end{corollary}
\noindent Observe that, in general, $\zeta(\varphi_S)$ is not an isomorphism, even when $\varphi_S$ is injective.

\section{Applications to the finite case}\label{FIN}
It would be desirable to  have a structure theorem for $\zeta(A)=\zeta(A,\mc{R}_A)$ in case this is finite. Observe that its finitess depends on the chosen relation, as for instance $\zeta(\QQ,=)$ is infinite, whereas $\zeta(\QQ,\sim)$ is finite. Clearly, when  $A$ is finite then $\zeta(A)$ is finite, indipendently of the given $\mc{R}_A$; moreover, we know by Proposition \ref{finestblend} and Remark \ref{coarse&quot} then when  $\zeta(A)$ is finite then  $\zeta(A,\sim)$ is finite. Therefore,  if we give, and we do, a characterization in the case $\mc{R}_A=\sim_A$, then we will have the largest family of rings for which $\zeta(A)$ might be finite.

For the rest of the section let us fix $\{\mc{R}_A\}=\{\sim_A\}$. First of all, notice that, by definition, $\zeta(A)$ is finite exactly when $A$ has finitely many principal ideals, which immediately implies that $A$ is artinian.

\begin{theorem}\label{finitetheta}
Let $A$ be a ring. Then, the following are equivalent:
\begin{enumerate}[1.]
    \item $\zeta(A)$ is finite;
    \item $A$ has finitely many ideals;
    \item $A\simeq A_1\times A_2$, where $A_1$ is an artinian PIR and  $A_2$ is a finite ring.
\end{enumerate}
\end{theorem}

\begin{proof}$1 \Rightarrow 2$: We already observed that $\mc{PI}_A$ is finite and $A$ artinian; since every element of $\mc{I}_A$ is a finite sum of principal ideals, the conclusion follows.

  \smallskip
  \noindent
  $2 \Rightarrow 3$: This is contained in \cite{Z14}, Theorem 3.3. Observe that $A$ is artinian and can be written as a finite product of artinian local rings, which all have a finite number of ideals. It is thus sufficient to show that each local factor of $A$ is either a PIR or a finite ring.  By contradiction, say that a factor $(B,\mf{m},k)$ of $A$ is infinite and not PIR.  It is well known that a local ring is finite if and only if $A$ is artinian with finite residue field and this is principal if and only if the $k$-vector space $\mf{m}/\mf{m}^2$ has dimension $>1$. We may thus conclude that $Q$ has infinitely many ideals, which is the sought after contradicion.
  
  \smallskip
  \noindent
$3 \Rightarrow 1$: By Corollary \ref{cpfsim} we have that $\zeta(A)\simeq\zeta(A_1)\times\zeta(A_2)$, where $\zeta(A_2)$ is finite. We only need to show that  $\zeta(A_1)$ is finite as well. Since an artinian PIR is a finite product of local artinian PIRs, we may assume that $A_1$ is local. Thus, every ideal of $A_1$ can be written as $(p^i)$, for some $p\in A_1$, and $p^n=0$ for some natural number $n\geq 1$ from which it follows that  $\mc{PI}_A$ is finite.
 \end{proof}

\noindent
The previous theorem, together with Corollary \ref{cpfsim}, yields a reduction to the case of local artinian rings, since we can reconstruct $\zeta(A)$ from the graphs which are associated with its local factors.

A family of graphs which turns out to be important to this end is staircases graphs, i.e. those whose adjacency matrix is an ascending staircase of 1's. In other words, a graph with  vertex set $V=\{0,\dots, k\}$ is called {\em staircase graph of index $k$}, and it is denoted by $SG_k$, when $$i \text{\, is adjacent  to\, } j   \text{\,\, if and only if \,\,} i+j\geq k.$$ 
\noindent
Observe that in $SG_k$ it holds that $\{k\}=N(0)\subset N(1)\subset\dots\subset N(k)=\{0,\dots,k\}$ and, thus,  $1=\deg(0)<\deg(1)<\dots<\deg(k)=k+1$.

\noindent
These graphs are uniquely determined, up to isomorphisms, by the degree of the vertices. In fact, in \cite{DJS21} Proposition 4.1 the following was proved: a graph $G=(\{v_0,\ldots,v_k\},E)$, with $k>0$,  such that $\deg(v_i)\in\{1,\ldots,k+1\}$ for each $i$ and $\deg(v_i)\neq \deg(v_j)$  whenever $i\neq j$, is isomorphic to  $SG_k$. Moreover, since  $SG_k\simeq \zeta(\ZZ/(p^k))$  and $\ZZ/(p^k)$ is a local ring, any $A$ such that $\zeta(A)\simeq SG_k$ is local as well, cf. \cite{DJS21}, Proposition 4.2 and Corollary 5.1.

Before we proceed with the next lemma, we observe the following. Fix $x\in A$; the map $m_x:\zeta(A)\rightarrow\zeta(A)$ defined as $(a)\mapsto (xa)$ is a graph morphism. Moreover, for every $z\in A$ we have that $(z)\in N((xa))\iff z\in \Ann(xa)\iff zx\in\Ann(a)\iff (zx)\in N((a))$.
Thus, if $a,b\in A$ are such that $N((a))\subseteq N((b))$ then $N((xa))\subseteq N((xb)).$

\begin{lemma}\label{SGmf}
Let $A$ be a ring, $a\in A$ and  $\varphi \,:\, SG_k\longrightarrow \zeta(A)$ a surjective strong morphism, for some integer $k\geq 0$; then, for all $0\leq i\leq j\leq k$, it holds that  $N(x\varphi(i))\subseteq N(x\varphi(j))$.
\end{lemma}
\begin{proof}
It is enough to prove that $N(\varphi(i))\subseteq N(\varphi(j))$. By hypothesis, every vertex in $N(\varphi(i))$ is the image of a vertex $l\in\{0,\dots, k\}$; since  $\varphi$ is a  comorphism, $l$ is adjacent to $i$ and  $i\leq j$ implies that  $l\in N(i)\subseteq N(j)$. Thus, $\varphi(l)\in N(\varphi(j))$.
\end{proof}

\begin{lemma}\label{propAnn1}
Let $A\neq\tf{0}$ be a ring such that  $\mc{D}(A)=\mc{N}(A)$; then, for all  $a,b\in\mc{D}^*(A)$ then  $\Ann(a)\cup \Ann(b)\subsetneq \Ann(ab)$. 
\end{lemma}
\begin{proof} It is well-known that  $\Ann(a)\cup\Ann(b)\subseteq\Ann(ab)$; therefore, let us assume by contradiction that there exist $a,b\in \mc{D}^*(A)$ such that $\Ann(a)\cup \Ann(b)=\Ann(ab)$ and let $n,m\geq 1$ be  the nilpotency indexes of $a$ and $b$, respectively.
By inducting on $k+l$, one can easily show first that $a^kb^l\neq 0$ for all  $k\in\{0,\dots,n-1\}$,  $l\in\{0,\dots, m-1\}$, and secondly that   
 $\Ann(a^kb^l)=\Ann(a^k)\cup \Ann(b^l)$, for all such $k,l$.

Now, if we let $\alpha=a^{n-1}$ and $\beta =b^{m-1}$, then  $\alpha,\, \beta,\, \alpha\beta \neq 0$ and  $\alpha^2=\beta^2=0$;  accordingly,  $(\alpha+\beta)\in \Ann(\alpha\beta)=\Ann(\alpha)\cup \Ann(\beta)$ and, thus, $\alpha\beta=0$, contradiction.
\end{proof}

\noindent
As an easy consequence, one characterizes local artinian rings by means  annihilator ideals, cf. \cite{DJS21}, Theorem  5.3.

\begin{corollary}\label{cla}
Let $A\neq\tf{0}$ be artinian; then $A$ is local if and only if for all $a,b\in\mc{D}^*(A)$ it holds that  $\Ann(a)\cup \Ann(b)\subsetneq \Ann(ab)$.
\end{corollary}

\begin{proof} If $A$ were not local, the Structure Theorem for artinian rings would yields that  $A\simeq A_1\times A_2$, for some rings $A_1,A_2\neq \tf{0}$; also, the elements $a=b=(1,0)$ would contradict the hypothesis.   Viceversa, when $A$ is local artinian one has $\mc{D}(A)=\mc{N}(A)=\mf{m}$, and the conclusion follows directly by the previous lemma.
\end{proof}

Recall that, when $A$ is local, one has that  $(a), (b)\in\mc{PI}_A$ are equal if and only if $a\approx b$.

\begin{lemma}\label{Thetastair}
Let $A$ be a local artinian  ring and  $\varphi \,:\, SG_k\longrightarrow \zeta(A)$ be an isomorphism for some $k>0$.
\begin{enumerate}[1.]
    \item Let $x\in A$; if there exist $0\leq i<j\leq k$ such that $\Ann(x\varphi(i))=\Ann(x\varphi(j))$, then $x\varphi(i)=x\varphi(j)=0;$
    \item $\varphi(i)=\varphi(1)^i$, for all  $i\in\{0,\dots, k\}.$
\end{enumerate}
\end{lemma}
\begin{proof} For all $i,j\in\{0,\dots, k\}$,\, if\,  $\varphi(i)=\varphi(j)$ then $N(\varphi(i))=N(\varphi(j))$, i.e. $\Ann(\varphi(i))=\Ann(\varphi(j))$, and the viceversa holds since $\varphi$ is an isomorphism. The graph $SG_k$ is staircase and, thus,  $i<j$ if and only if $N(i)\subset N(j)$; this is equivalent to  $N(\varphi(i))\subset N(\varphi(j))$, which in turn means  $\Ann(\varphi(i))\subset\Ann(\varphi(j))$. For all  $i\in\{0,\dots,k\}$, let $\varphi(i)=(a_i).$

  \smallskip
  \noindent
1. Assume that $\Ann(xa_i)=\Ann(xa_j)$; then  $(xa_i)=(xa_j)$ and there exists $u\in A^*$ such that $a_i-ua_j\in\Ann(x)$. Moreover, since $i<j$, we have that $\Ann(a_i)\subset\Ann(a_j)$ and there exists $y\in\Ann(a_j)\smallsetminus\Ann(a_i)$; thus, $(a_i-ua_j)y=a_iy\neq 0$ and it follows that $a_i-ua_j\not\in\Ann(y)$.
  
    Since $\varphi^{-1}\big((x)\big),\varphi^{-1}\big((y)\big)\in V(SG_k)$, we have that either  $N(\varphi^{-1}\big((x)\big))\subseteq N(\varphi^{-1}\big((y)\big))$,\, or\,  $N(\varphi^{-1}\big((x)\big))\supseteq N(\varphi^{-1}\big((y)\big))$ and, accordingly, either $\Ann(x)\subseteq\Ann(y)$\, or\,  $\Ann(y)$ $\subseteq\Ann(x)$. Since  $a_i-ua_j\in\Ann(x)\smallsetminus\Ann(y)$, we must have $\Ann(y)\subset\Ann(x)$ and, by our choice of $y$, it follows that  $0=x(a_j)=x(a_i)$.

    \smallskip
    \noindent
    2. The cases $i=0,1$ are easy: $(a_1)\!=\!(a_1)^1$ and, since $\Ann(a_0)\!=0\!=\!\Ann(1)$, it follows that  $(a_0)\!=\!(1)\!=\!(a_1)^0$. Thus, if $k=1$ we are done and may assume $k>1$. We are going to prove that $(a_{i+1})=(a_1)(a_i)$ for all $1\leq i<k$, and that implies the conclusion.

    First of all, observe that  $N(j)\!\neq\!N(0)$ for all $j\!\in\!\{1,\dots, k\!-\!1\}$. This implies that  $\Ann(a_j)\neq\Ann(a_0)=0$ and, thus,  $a_j\in \mc{D}(A)$; moreover,  $(a_j)\!\neq\!(a_k)=0$,\, i.e.\,   $a_j\in\mc{D}^*(A)$. By Corollary \ref{cla}, we have that  $\Ann(a_1a_j)\!\supset\!\Ann(a_j)$ and  $N((a_1a_j))\!\supset\!N((a_j))$; therefore  $\deg((a_1a_j))\!>\!\deg((a_j))\!=\!j+1$ for all  $j\!\in\!\{1,\dots, k\!-\!1\}$.
    
    Now take $j, l\in\{0,\dots,k\}$ with $j\neq l$; by 1, if $\Ann(a_1a_j)=\Ann(a_1a_l)$ then $a_1a_j=a_1a_l=0$ and, since $\varphi$ is an isomorphism, the only possible cases are when $j,l\in\{k\!-\!1,k\}$.  Moreover, by Lemma \ref{SGmf} we get that  $N((a_1a_i))\!\subset\! N((a_1a_{i+1}))\!\subset\!N((a_1a_{i+2}))\!\subset\dots\subset\!N((a_1a_{k-2}))\!\subset\! N((a_1a_{k-1}))=N((a_1a_k))=\mc{PI}_A$, where the last equality is due to the fact that  $a_k=0$. Taking the degrees in  the above chain of inclusions, we obtain $\deg((a_1a_i))\, <\, \deg((a_1a_{i+1}))\, <\, \deg((a_1a_{i+2}))\,< \,\dots\,<\,\deg((a_1a_{k-2}))\, < \,\deg((a_1a_{k-1}))=\deg((a_1a_k))=k+1$. We know that $\deg((a_1a_i))\!>\!i+1$; on the other hand  $\deg((a_1a_i))\leq i+2$, otherwise we would have  $\deg((a_1a_{k-1}))\!>\!i+2+(k\!-\!1-i)=k+1$, which is a contradiction. Therefore, $(a_1a_i)$ has degree $i+2$ and, therefore,  $(a_1a_i)=(a_{i+1})$, as desired.

\end{proof}

\begin{proposition}\label{PIRloc}
Let $(A,\mf{m})$ be a local artinian ring. Then, $A$ is PIR if and only if $\zeta(A)\simeq SG_k$, where $k$ is the nilpotency index of  $\mf{m}$.
\end{proposition}
\begin{proof}
  Since $A$ is PIR, it is well-known that every  ideal $I\subseteq A$ can be written as $I=\mf{m}^i$ and $\mf{m}^i$ is adjacent to $\mf{m}^j$ if and only if $i+j\geq k$, where  $k$ is the nilpotency index of  $\mf{m}$; it is easy to conclude that the map  $\varphi \,:\, SG_k\longrightarrow \zeta(A)$ defined by $\varphi(i)=\mf{m}^i$ is an isomorphism.
  
Viceversa, by Lemma \ref{Thetastair}.2, every principal ideal is a power of a fixed principal ideal $(x)$; since $A$ is noetherian, it follows that every ideal is a power of $(x)$ and that  $A$ is PIR.
\end{proof}

\begin{example}
  Let $A= \ZZ/(4)[x]/(2x,x^2-2)$ and $B=\mathbb{F}_2[x,y]/(x^2,xy,y^2)$. Then, $A$ is a local ring which is PIR by Proposition \ref{PIRloc}, since  $\zeta(A)\simeq SG_3$; since  $\zeta(\mathbb{F}_2[x,y]/(x^2,xy,y^2))$ has four distinct vertices with a loop, namely  $(\bar{x})$, $(\bar{y})$, $(\bar{x}+\bar{y})$ and $(\bar{0})$,  in contrast to $SG_3$ which has only two vertices with a loop.
  
\begin{center}
    \includegraphics[scale=0.3]{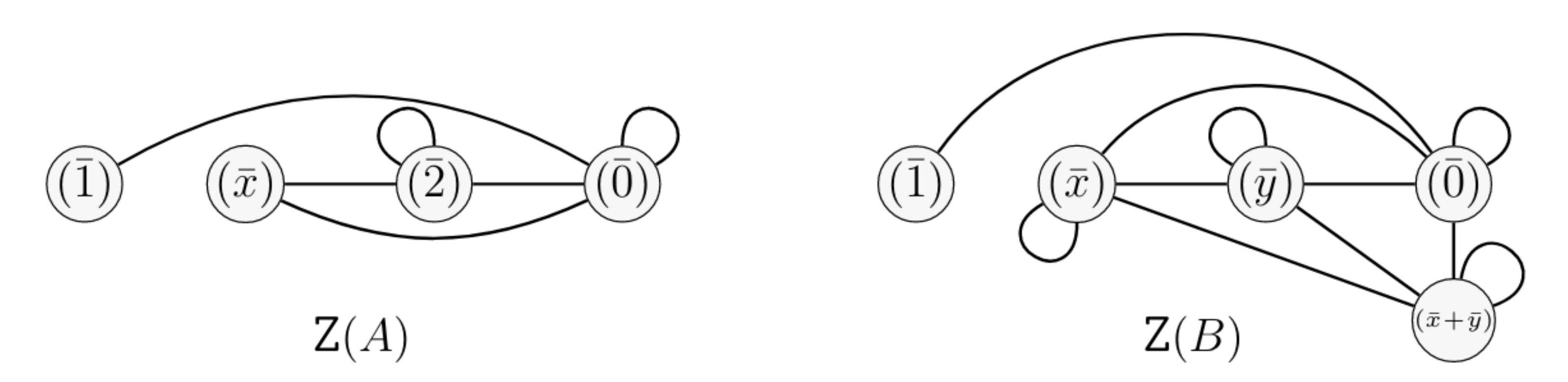}
\end{center}
\end{example}

\begin{theorem}\label{thm: caratterizzazione PIR}
An artinian ring  $A$ is PIR if and only if $\zeta(A)$ is isomorphic to a finite product of staircases graphs.
\end{theorem}
\begin{proof}
If $A$ is PIR, then the conclusion follows by a straightforward application of Proposition \ref{PIRloc} and Corollary \ref{cpfsim}.
  
Viceversa, if $\zeta(A)\simeq\bigtimes\limits_{i=1}^n SG_{k_i}$, then  $\zeta(A)\simeq \bigtimes\limits_{i=1}^n\zeta(\ZZ/(p_i)^{k_i})$. Since an artinian ring is total,  Corollary \ref{totale} implies that  $A\simeq \bigtimes\limits_{i=1}^n A_i$ where $\zeta(A_i)\simeq SG_{k_i}$.
Each $A_i$ is necessarily artinian; if $A_i$ were not local for a certain $i$, we would have  $A_i\simeq  B_1\times B_2$ with $B_1,B_2\not\simeq\tf{0}$, $\zeta(\ZZ/(p_i)^{k_i})\simeq\zeta(A_i)\simeq\zeta(B_1)\times\zeta(B_2)$ by Corollary \ref{cpfsim}, and another application of Corollary \ref{totale} would imply that also $\ZZ/(p_i)^{k_i}$ can be factorized, contradiction. Finally, each of the $A_i$ is PIR, as yielded by Proposition \ref{PIRloc}.
\end{proof}
\begin{remark}
  Keeping in mind Theorem \ref{finitetheta}, in order to complete the picture, one should understand the structure of $\zeta(A)$ when  $A$ is a finite ring, which is not PIR. This does not seem to be an easy task, since if  $A$ and $B$ are rings such that the posets  $\mc{I}_A$ and $\mc{I}_B$ are isomorphic, it is not true in general that  $\zeta(A)\simeq\zeta(B)$. For instance  consider  $A=\mb{F}_2[x,y]/(x^2,y^2)$\, and \,$B=\mb{F}_2[x,y]/(xy,x^2-y^2)$; then $\mc{I}_A\simeq \mc{I}_B$ as posets but every vertex in $\zeta(A)$ has a loop whereas $(\bar{x})\in\mc{PI}_B$ has none; therefore $\zeta(A)\not\simeq\zeta(B)$. 
  
\begin{center}
    \includegraphics[scale=0.25]{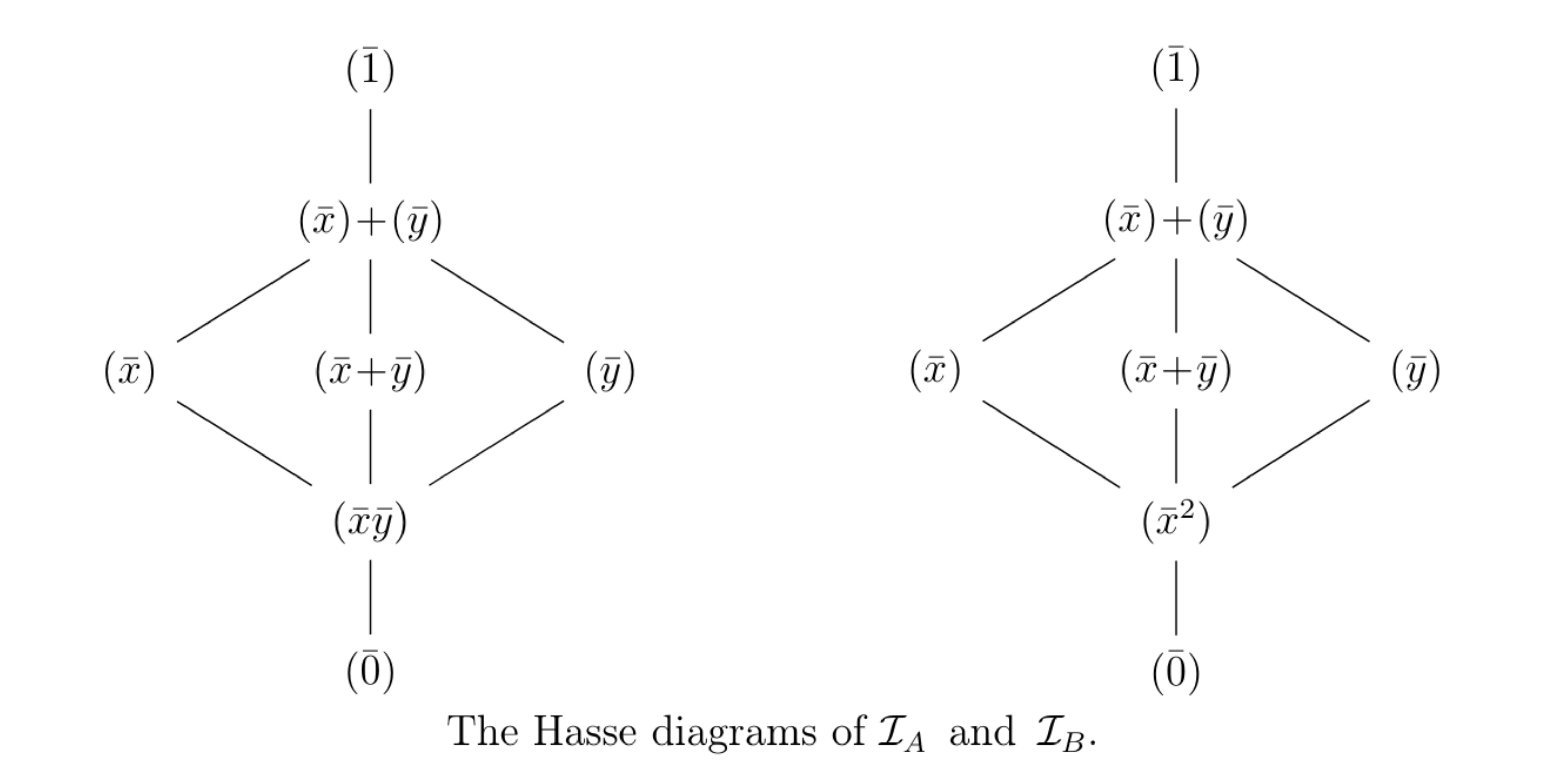}
\end{center}
\end{remark}


\begin{thebibliography}{CSSS12}

\bibitem[AL12]{AL12}
  D. F. Anderson, J. D. LaGrange.
  \newblock{Commutative Boolean monoids, reduced rings, and the compressed zero-divisor graph.}
  \newblock{\it Journal of Pure and Applied Algebra}, 216 (2012), 1626-1636.

\bibitem[AL99]{AL99}
  D. F. Anderson, P. S. Livingston.
  \newblock{The Zero-Divisor Graph of a Commutative Ring.}
  \newblock {\em Journal of Algebra}, 217 (1999), 434-447.

\bibitem[ALS03]{ALS03}
D. F. Anderson, R. Levy, J. Shapiro.
\newblock {Zero-divisor Graphs, von Neumann Regular Rings, and Boolean Algebras.}
\newblock {\em Journal of Pure and Applied Algebra},  180 (2003), 221-241.

\bibitem[AM04]{AM04}
  S. Akbari, A. Mohammadian.
  \newblock{On the Zero-Divisor Graph of a Commutative Ring.}
  \newblock{\em Journal of Algebra}, 274 (2004), 847-855.

\bibitem[AMY03]{AMY03}
  S. Akbari, H. R. Maimani, S. Yassemi.
  \newblock {When a zero-divisor graph is planar or a complete r-partite graph.}
  \newblock {\em Journal of Algebra}, 270 (2003), 169--180.

\bibitem[B88]{B88}
I. Beck.
\newblock {Coloring of Commutative Rings.}
\newblock {\em Journal of Algebra}, 116 (1988), 208--226.

\bibitem[B94]{B94}
F. Borceux.
\newblock{Handbook of Categorical Algebra Vol. 1: Basic Category Theory.}
\newblock{Encyclopedia of Mathematics and its Applications 50}, Cambridge University Press (1994).

\bibitem[CSSS12]{CSSS12}
J. Coykendall, S. Sather-Wagstaff, L. Sheppardson, S. Spiroff.
\newblock {On Zero Divisor Graphs.}
\newblock In {\em Progress in Commutative Algebra 2}, De Gruyter (2012), 241--299.

\bibitem[DJS21]{DJS21}
A. \DJ uri\'c, S. Jev\dj eni\'c, N. Stopar. 
\newblock {Categorial Properties of Compressed Zero-Divisor Graphs of Finite Commutative Rings.}
\newblock {\em Journal of Algebra and Its Applications}, 20 (2021), Paper No. 2150069, 16 pp.


\bibitem[K11]{K11}
  U. Knauer.
  \newblock {\em Algebraic Graph Theory: Morphisms, Monoids and Matrices},
  \newblock {De Gruyter, Berlin (2011)}.

 
\bibitem[M02]{M02}
  S. B. Mulay.
\newblock {Cycles and Symmetries of Zero-Divisors.}
\newblock {\em Communications in Algebra}, 30 (2002), 3533--3558.

\bibitem[SW11]{SW11}
  S. Spiroff and C. Wickman.
  \newblock{A Zero Divisor Graph Determined by Equivalence Classes of Zero Divisors.}
  \newblock{\em Communications in Algebra}, 39 (2011), 2338--2348.

\bibitem[SZ22]{SZ22}
E. Sbarra, M. Zanardo
 \newblock{Some graph-theoretic properties of zero-divisor functors}
 \newblock{\em In preparation}, (2022).

  
\bibitem[Z14]{Z14}
  M. B. J. Zwann.
\newblock {On commutative rings with only finitely many ideals.}
\newblock {Thesis, Mathematisch Instituut, Universiteit Leiden.}
\newblock Available at
          {\url{https://www.math.leidenuniv.nl/scripties/ZwaanBach.pdf}}.

\end{thebibliography}
\end{document}